\documentclass[reqno]{amsart}
\usepackage{amsthm}
\usepackage{times,amssymb,amsmath,amsfonts,bbm,mathrsfs,color}
\usepackage[mathscr]{eucal}
\usepackage[top=1.5in,bottom=1.43in,left=1.25in,right=1.25in]{geometry}
\usepackage{hyperref} 
\allowdisplaybreaks
\newcommand\nc\newcommand
\nc\bfa{{\boldsymbol a}}\nc\bfA{{\boldsymbol A}}\nc\cA{{\mathscr A}}
\nc\bfb{{\boldsymbol b}}\nc\bfB{{\boldsymbol B}}\nc\cB{{\mathscr B}}
\nc\bfc{{\boldsymbol c}}\nc\bfC{{\boldsymbol C}}\nc\cC{{\mathscr C}}
\nc\bfd{{\boldsymbol d}}\nc\bfD{{\boldsymbol D}}\nc\cD{{\mathscr D}}
\nc\bfe{{\boldsymbol e}}\nc\bfE{{\boldsymbol E}}\nc\cE{{\mathscr E}}
\nc\bff{{\boldsymbol f}}\nc\bfF{{\boldsymbol F}}\nc\cF{{\mathscr F}}
\nc\bfg{{\boldsymbol g}}\nc\bfG{{\boldsymbol G}}\nc\cG{{\mathscr G}}
\nc\bfh{{\boldsymbol h}}\nc\bfH{{\boldsymbol H}}\nc\cH{{\mathscr H}}
\nc\bfi{{\boldsymbol i}}\nc\bfI{{\boldsymbol I}}\nc\cI{{\mathcal I}}
\nc\bfj{{\boldsymbol j}}\nc\bfJ{{\boldsymbol J}}\nc\cJ{{\mathscr J}}
\nc\bfk{{\boldsymbol k}}\nc\bfK{{\boldsymbol K}}\nc\cK{{\mathscr K}}
\nc\bfl{{\boldsymbol l}}\nc\bfL{{\boldsymbol L}}\nc\cL{{\mathscr L}}
\nc\bfm{{\boldsymbol m}}\nc\bfM{{\boldsymbol M}}\nc{\cM}{{\mathscr M}}
\nc\bfn{{\boldsymbol n}}\nc\bfN{{\boldsymbol N}}\nc\cN{{\mathscr N}}
\nc\bfo{{\boldsymbol o}}\nc\bfO{{\boldsymbol O}}\nc\cO{{\mathscr O}}
\nc\bfp{{\boldsymbol p}}\nc\bfP{{\boldsymbol P}}\nc\cP{{\mathscr P}}
\nc\bfq{{\boldsymbol q}}\nc\bfQ{{\boldsymbol Q}}\nc\cQ{{\mathscr Q}}
\nc\bfr{{\boldsymbol r}}\nc\bfR{{\boldsymbol R}}\nc\cR{{\mathscr R}}
\nc\bfs{{\boldsymbol s}}\nc\bfS{{\boldsymbol S}}\nc\cS{{\mathscr S}}
\nc\bft{{\boldsymbol t}}\nc\bfT{{\boldsymbol T}}\nc\cT{{\mathscr T}}
\nc\bfu{{\boldsymbol u}}\nc\bfU{{\boldsymbol U}}\nc\cU{{\mathscr U}}
\nc\bfv{{\boldsymbol v}}\nc\bfV{{\boldsymbol V}}\nc\cV{{\mathscr V}}
\nc\bfw{{\boldsymbol w}}\nc\bfW{{\boldsymbol W}}\nc\cW{{\mathscr W}}
\nc\bfx{{\boldsymbol x}}\nc\bfX{{\boldsymbol X}}\nc\cX{{\mathscr X}}
\nc\bfy{{\boldsymbol y}}\nc\bfY{{\boldsymbol Y}}\nc\cY{{\mathscr Y}}
\nc\bfz{{\boldsymbol z}}\nc\bfZ{{\boldsymbol Z}}\nc\cZ{{\mathscr Z}}
\nc{\bb}{{\mathbbm{1}}}
\nc\reals{{\mathbb R}}
\nc\EE{{\mathbb E}}
\nc{\blue}[1]{{\color{blue}{#1}}}
\nc{\red}[1]{\textcolor{red}{#1}}
\nc{\comment}[1]{\textcolor{red}{(#1)}}

\newtheorem{theorem}{Theorem}[section]
\newtheorem{lemma}[theorem]{Lemma}

\newtheorem{proposition}[theorem]{Proposition}

\theoremstyle{remark}

\DeclareMathOperator{\rank}{rank}

\DeclareMathOperator{\diam}{diam}

\renewcommand{\hat}{\widehat}

\begin{document}
		\title{Bounds for discrepancies in the Hamming space}
		\author[A. Barg]{Alexander Barg$^{1}$}
		\thanks{\noindent
		$^1$ Department of ECE and Institute for Systems Research, University of Maryland, College Park, MD 20742, USA and Inst. for Probl. Inform. Trans., Moscow, Russia. Email: abarg@umd.edu. Research of this author was partially supported by NSF grants CCF1618603 and CCF1814487.}
		\author[M. Skriganov]{Maxim Skriganov$^{2}$}
		\thanks{\hspace*{0in}$^2$
St.~Petersburg Department of Steklov Institute of Mathematics, Russian Academy of Sciences, nab. Fontanki 27, St.~Petersburg, 191023, Russia. Email: maksim88138813@mail.ru.}

 \begin{abstract}  
We derive bounds for the ball $L_p$-discrepancies in the Hamming space for 
$0<p<\infty$ and $p=\infty$. Sharp estimates of discrepancies have been 
obtained for many spaces such as the Euclidean spheres and more general compact 
Riemannian manifolds. 
In the present paper, we show that the 
behavior of discrepancies in the Hamming space differs fundamentally because 
the volume of the ball in this space 
depends on its radius exponentially while such a dependence for the 
Riemannian manifolds is polynomial.
 \end{abstract}
 \maketitle
\section{Introduction} 
\subsection{Basic definitions.}
Let $\cX_n=\{0,1\}^n$ be the binary Hamming space which can be also thought of 
as a linear space $\mathbb F^n_2$ over the finite field $\mathbb F_2$. The 
cardinality $|\cX_n|=2^n$.
Denote 
by 
$B(x,t)$ the ball with center at $x\in\cX_n$ and radius $t\ge 0$, i.e., the set 
of all
points $y\in \cX_n$ with $d(x,y)\le t,$ where $d(x,y)$ is the Hamming distance. 
The volume of the ball $v(t):=|B(x,t)|=\sum_{i=0}^t\binom ni$ is independent of 
$x\in\cX_n$. It is convenient to assume that $B(x,t)=\emptyset$ and $v(t)=0$
for $t<0$, and $B(x,t)=\cX_n$ and $v(t)=2^n$ for $t>n$.

For an $N$-point subset $Z_N\subset {\cX_n}$ and a ball $B(y,t)$ define the 
local discrepancy as follows:
   \begin{align}\label{eq:local}
D(Z_N,y,t)=|B(y,t)\cap Z_N|-N\, 2^{-n} v(t).
    \end{align}
We note that $D(Z_N,y,n)=0$ for any $Z_N,y,$ and thus below we limit ourselves to the values $0\le t\le n-1.$
Define the weighted $L_p$-discrepancy by
       \begin{align}\label{eq:Dp}
D_p(G,Z_N)=\Bigl(\,\sum\nolimits^{n-1}_{t=0}g_t \sum\nolimits_{y\in\cX_n}2^{-n} 
              |D(Z_N,y,t)|^p\,\Bigr)^{1/p},\quad 0<p<\infty\, ,
     \end{align}
where $G=(g_0,\dots,g_{n-1})$ is a vector of nonnegative weights 
normalized by
\begin{equation}\label{eq:norm}
\sum\nolimits_{t=0}^{n-1} g_t=1.
\end{equation}
With such a normalization, we have
\begin{equation}\label{eq:incr}
D_p (G,Z_N)\,\le\, D_q (G,Z_N)\quad 0<p< q <\infty \, .
\end{equation}

The $L_{\infty}$-discrepancy is defined by
\begin{equation}
D_\infty(I,Z_N)=\max_{t\in I}\max_{y\in \cX_n} |D(Z_N,y,t)|\, ,
\label{eq:Dinfty}
\end{equation}  
where $I\subseteq\{0,\dots,{n-1}\}$ is a subset of the set of the radii.

We also introduce the following extremal discrepancies  
  \begin{equation*}
  D_p(G,n,N)=\min_{Z_N\subset\cX_n} D_p(G,Z_N)\, , \quad 0<p < \infty\, ,
  \end{equation*}
and 
\begin{equation*}
D_\infty(I,n,N)=\min_{Z_N\subset\cX_n} D_\infty (I,Z_N)\, .
\end{equation*}    
These quantities can be thought of as geometric characteristics of 
the Hamming space.

It is useful to keep in mind the following simple observations:

(i) If $Z_N^c=\cX_n\setminus Z_N$ is the complement of $Z_N\subseteq\cX_n$, 
then 
$
D(Z_N,y,t)=-D(Z_N^c,y,t) \, ,
$
and we have
\begin{equation*}
D_p(G,Z_N)=D_p(G,Z_N^c)\quad \mbox{and} \quad D_p(G,n,N)=D_p(G,n,2^n-N) \, ,
\end{equation*}
for all $0<p\le\infty$.
Hence, generally it suffices to consider only subsets $Z_N$ with $N\le 2^{n-1}$.
Together with results of \cite{Barg2020} on quadratic 
discrepancies this gives rise to the next claim:
	Let $Z_N$ be a perfect code in $\cX_n$, then the set $Z_N^c$ attains the 
	minimum 
	value $D_2(G_1,n,2^n-N)$,
where $G_1=(1/n,1/n,\dots,1/n)$.
For instance, for $n=2^m-1$ and $N=2^n(1-2^{-m}), m\ge 2$ the code $Z_N$ formed 
of spheres of radius one around the codewords of the Hamming code
(i.e., the union of the $n$ cosets of the Hamming code) is a minimizer of 
quadratic discrepancy. Another family of minimizers is given by
$\cX_n\backslash\{y,\bar y\}$ for any $y\in\cX_n$, 
where $\bar y:=1^n+y$ is a point antipodal to $y$ and $1^n\in\cX_n$ denotes the 
all-ones vector. Some 
other examples can be also given; see \cite{Barg2020}.
For the reader's convenience, we emphasize that the quadratic 
discrepancy $D^{L_2}(Z_N)$ 
in \cite{Barg2020}	is
related with our definition \eqref{eq:Dp} by 
$D^{L_2}(Z_N)=2^n n\, N^{-2}(D(G_1,Z_N))^2$.

(ii) Without loss of generality we can restrict the range of summation on $t$ 
in \eqref{eq:Dp} from $\{0,\dots,n\}$ to $\{0,\dots,\nu\},$ where $\nu=\lfloor 
(n-1)/2\rfloor,$ limiting ourselves to a \emph{half} of the full range. 
More precisely, we have 
\begin{equation*}
D_p(G,Z_N)=D_p(G^{*},Z_N)\quad \mbox{and} \quad D_p(G,n,N)=D_p(G^{*},n,N) \, ,
\end{equation*}
where $G^{*}=(g_1^{*},\dots ,g_{\nu}^{*})$ with $g_t^{*}=g_t +g_{n-t+1}$.

Indeed, notice that 
$B(y,t)=\cX_n \setminus B(\bar y, n-1-t)$, and therefore
$D(Z_N,y,t)=D(Z_N,\bar y,n-1-t)$. Also, obviously,
\begin{align*}
\sum\nolimits_{y\in\cX_n} |D(Z_N,\bar y,t)|^p=\sum\nolimits_{y\in\cX_n} 
|D(Z_N,y,t)|^p\, ,
\end{align*}
and thus 
\begin{align*}
D_p(G,Z_N)&=\Bigl(\sum\nolimits^{\nu}_{t=0}\Bigl( g_t 
2^{-n}\sum\nolimits_{y\in\cX_n} |D(Z_N,y,t)|^p+
g_{n-1-t}2^{-n}\sum\nolimits_{y\in\cX_n} |D(Z_N,\bar 
y,t)|^p\Bigr)\Bigr)^{1/p}\\
&=\Bigl(\,2^{-n}\sum\nolimits^{\nu}_{t=0}(g_t+g_{n-t-1}) 
\sum\nolimits_{y\in\cX_n} 
|D(Z_N,y,t)|^p\,\Bigr)^{1/p}.
\end{align*}
We conclude that limiting the summation range of $t$ amounts to changing the 
weights in definition \eqref{eq:Dp}.
Similar arguments hold true for the $L_{\infty}$-discrepancy \eqref{eq:Dinfty}.


\subsection{Earlier results.}
Discrepancies in compact metric measure spaces have been studied for a long time, starting with basic results in the theory of uniform distributions \cite{Beck1984,Beck1987,Matousek1999}. 
In particular, quadratic discrepancy of finite subsets of the Euclidean sphere is related to the structure of the distances in the subset through a well-known identity called Stolarsky's invariance principle \cite{Stolarsky1973}. 
Stolarsky's identity expresses the $L_2$-discrepancy of a spherical set as a difference between the average distance on the sphere and the 
average distance in the set. Recently it has been a subject of renewed attention in the literature. In particular, 
papers \cite{Brauchart2013,Skriganov2017,Bilyk2018} gave new, simplified proofs of Stolarsky's invariance, while \cite{Skriganov2020}
extended Stolarsky's principle to projective spaces and derived asymptotically tight estimates of discrepancy. Sharp bounds on 
quadratic discrepancy
were obtained in \cite{BCCGT2019,BG2015,Skriganov2017, Skriganov2019}. Finally, paper \cite{skriganov2018bounds} 
introduced new asymptotic upper bounds on $L_p$-discrepancies of finite sets in compact metric measure spaces.

A recent paper \cite{Barg2020} 
initiated the study of Stolarsky's invariance in finite metric spaces, deriving 
an explicit form of the invariance
principle in the Hamming space $\cX_n$ as well as bounds on the quadratic discrepancy of subsets (codes) in $\cX_n.$ Explicit formulas
were obtained for the uniform weights $G_1=(1/n,1/n,\dots,1/n)$. 
Namely, let $x,y\in \cX_n$ be two points with $d(x,y)=w.$ Define
  $$
  \lambda(x,y)=\lambda(w):=2^{n-w}w\binom{w-1}{\lceil\frac w2\rceil-1}, \quad w=0,\dots,n.
  $$
As shown in \cite[Eq.~(23)]{Barg2020},
Stolarsky's identity for $Z_N\subset\cX_n$ can be written in the following form:
  \begin{equation}\label{eq:Ex}
  {2^n}n D_2(G_1,Z_N)^2=\frac{nN^2}{2^{n+1}}\binom{2n}n-\sum_{i,j=1}^N 
  \lambda(d(z_i,z_j)).
  \end{equation}
Using this representation, \cite[Cor.5.3, Thm.5.5]{Barg2020}  further showed that
  $$
   c\,n^{-3/4}\,N^{1/2}\Big(1-\frac{N}{2^n}\Big)^{1/2}\le D_2(G_1,n,N)\le 
   C\,n^{-1/4}\,N^{1/2}\, ,
  $$
where $c,C$ are some universal constants. Here the upper bound is proved by random choice and the lower bound by linear programming. 
The method of linear programming, well known in coding theory \cite{del72b,lev98}, is applicable to the problem of bounding the 
quadratic discrepancy because it can be expressed as an energy functional on the code with potential given by $\lambda.$
Moreover, there exist sequences of subsets (codes)
$Z_N\subset \cX_n, n=2^m-1$ whose quadratic discrepancy meets the lower bound. 
Observe also that if $N=o(2^n),$ then the bounds differ only by a factor of $n$: for example, if $N\simeq 2^{\alpha n},\, 0<\alpha <1$, 
then
\begin{equation}\label{eq:LP}
N^{1/2}\, (\log N)^{-3/4}\,\lesssim D_2(G_1,n,N)\lesssim\, N^{1/2} \,(\log 
N)^{-1/4}\, ,
\end{equation}

In this short paper we develop the results of 
\cite{Barg2020}, proving bounds on $D_p(G,n,N), p\in(0,\infty].$  We also 
consider a restricted version of the discrepancy $D_p(G,Z_N),$ limiting 
ourselves to the case of hemispheres in $\cX_n.$ In other words, we take local
discrepancy for $t=(n-1)/2$ in \eqref{eq:local} ($n$ odd) and average its value over the centers of the balls. 
For the case of the Euclidean sphere, quadratic discrepancy for hemispheres was 
previously studied in \cite{Bilyk2018, Skriganov2019}, which established a 
version of Stolarsky's invariance for this case.

\section{Bounds on \texorpdfstring{${D_p(G,n,N)}$}{}}
We are interested in universal bounds for discrepancies 
\eqref{eq:Dp}--\eqref{eq:Dinfty} for given $n,N$ and $p\in 
(0,\infty]$ without 
accounting for the structure of the subset. For the case of 
finite subsets in compact Riemannian manifolds this problem was recently 
studied in \cite{skriganov2018bounds}, and we draw on the approach of this 
paper in the derivations below.

\subsection{The case \texorpdfstring{${0<p<\infty}$}{}}
We shall consider random subsets $Z_N\subset\cX_n$, using the following standard result to handle discrepancies of such subsets.
\begin{lemma}[Marcinkiewicz--Zygmund inequality; \cite{CT1988}, Sec.10.3]
	Let $\zeta_j,\, j\in J,\,\, |J|<\infty,$ be a finite collection of 
	real-valued independent random variables 
	with expectations 
	$\EE\,\zeta_j =0$ and $\EE\,\zeta_j^2<\infty\, , j\in J$. Then, we 
	have 
	\begin{equation*}
	\EE\, |\,\sum\nolimits_{j\in J}\zeta_j\,|^p\,\le\, 2^p\, (p+1)^{p/2}\,\EE\,
	(\,\sum\nolimits_{j\in J}\zeta^2_j \,)^{p/2}, 
	\quad 1\le p<\infty\, .
	\end{equation*}
\end{lemma}

In our first result we construct a random subset $Z_N$ by uniform random choice. Later we will refine this procedure, obtaining a more 
precise bound on $D_p.$
    \begin{theorem}\label{thm:random}
	For all $N\le 2^{n-1}$, we have 
    \begin{equation}\label{eq:random}
D_p(G,n,N)\, \le\, \begin{cases} 
\, 2(p+1)^{1/2}\,N^{1/2}&\text{for\, $1\le p<\infty$,}\\[.05in]
\, 2^{3/2}\,N^{1/2}&\text{for\, $0<p<1$.}
\end{cases}
\end{equation} 
\end{theorem} 

{\em Remark 2.1.} Bounds of the type \eqref{eq:random} 
hold true for arbitrary compact metric measure spaces. Theorem \ref{thm:random} is given here to compare it with Theorem 
\ref{thm:random2} below. Notice also that the upper bound \eqref{eq:LP} is better 
than \eqref{eq:random} with $p=2$ and $G=G_1$ by a logarithmic factor. Such an improvement
is obtained in \cite{Barg2020} because of the explicit formula \eqref{eq:Ex} 
for the quadratic discrepancy with the uniform weights $G_1$.
\begin{proof}
Choose a subset $Z_N$ by selecting the points $\{z_i\}_1^N$ independently and uniformly in $\cX_n$. 
The probability that such a point falls into a subset
$\cE\in\cX_n$  equals to $|\cE|/|\cX_n|.$ Therefore, for the local discrepancy  \eqref{eq:local}
of this random subset $Z_N$ we have
\begin{equation}\label{eq:sum10}
D(Z_N,y,t)=\sum\nolimits^N_{i=1} \zeta_i(y,t)\, ,
\end{equation}
where
\begin{equation*}
\zeta_i(y,t)=\bb_{B(y,t)}(z_i)-\frac{v(t)}{|\cX_n|} \, ,
\end{equation*}
where $\bb_{\cE}$ is the indicator function of a subset $\cE\subseteq\cX_n$. The quantities $\zeta_i(y,t)$ are independent random 
variables that satisfy $|\zeta_i (y,t)|\le 1$ and $\EE\,\zeta_i (y,t)=0$. 

Applying the Marcinkiewicz--Zygmund inequality to the sum \eqref{eq:sum10}, we obtain
\begin{equation*}
\EE\, |D(Z_N,y,t)|^p\le 2^p\, (p+1)^{\,p/2}\,N^{\,p/2},
\quad 1\le p<\infty\, ,
\end{equation*}
and, therefore, in view of \eqref{eq:norm},
\begin{equation*}
\EE\, D(G,Z_N)^p\le 2^p\, (p+1)^{\,p/2}
\, N^{\,p/2},
\quad 1\le p<\infty\,.
\end{equation*}
Thus, there exists a subset 
$Z_N=Z_N(p)\subset \cX_n,\, 1\le p<\infty,$ whose discrepancy is bounded above 
as in this inequality. For $0<p<1$, in view of \eqref{eq:incr}, we can put 
$Z_{{N}}(p)=Z_{{N}}(1)$ to complete the proof.
\end{proof}	

In some situations the bound of this theorem can be improved relying on the method of {\em jittered} (or {\em stratified}) sampling, which
uses a partition of the metric space into subsets of small diameter and equal 
volume. This idea goes back to classical works on discrepancy theory 
\cite[pp.237-240]{Beck1984, Beck1987} and it was used more recently in 
\cite{BilykLacey2017,BCCGT2019,BGKZ2020} for the case of the 
Euclidean sphere and in \cite{skriganov2018bounds} for general metric spaces. Below we follow the approach of \cite{skriganov2018bounds}. 
In the case of the Hamming space the natural way to proceed is to partition 
$\cX_n$ into sub-hypercubes of a fixed dimension.

In our analysis
bounds on the volume of 
ball $v(t)$ are crucial.  
For large $n$ and $t=\lambda n\, ,0\le\lambda\le 1$, the 
well-known bound on $v(t)$ (cf. \cite[p.~310]{mac91}), can be written in the form
\begin{align}\label{eq:vol}
v(\lambda n)\le 2^{nH(\lambda)}\, ,
\end{align}
where
\begin{equation}\label{eq:H}
H(\lambda)=
\begin{cases}  
h(\lambda),&\text{if $0\le\lambda\le 1/2$},\\
1,&\text{if $1/2<\lambda\le 1$}\, ,
\end{cases}
\end{equation}
and $h(\lambda)=-\lambda\log_2\lambda-(1-\lambda)\log_2(1-\lambda)$ is the 
standard binary entropy, and in general, the bound \eqref{eq:vol} can not be 
improved.
Formally speaking, the statement \eqref{eq:vol} requires $\lambda n$ be 
integer, but this 
does not matter for the asymptotic arguments that we employ.  

\begin{theorem}\label{thm:random2}
	Let 
	$N=2^{\alpha n},\, 0<\alpha <1,$ be a power of  $2$.
	Suppose that the weights
	$g_t=0$ for $t>\beta n,\, 0<\beta<1/2$. Then  
	\begin{equation}\label{eq:hc}
	D_p(G,n,N)\, \le \, \begin{cases}
	\, 2(p+1)^{1/2}\,\, N^{(1-\kappa)/2},&\text{for $1\le p<\infty$\, ,}\\[.05in]
	\, 2^{3/2}\,N^{(1-\kappa)/2},&\text{for $0<p<1$\, ,}
	\end{cases} 
	\end{equation}
	where
\begin{equation}\label{eq:kappa}
\kappa=\kappa (\alpha,\beta)=\frac{1-H(1+\beta -\alpha)}{\alpha}\,\ge 0 \, .
\end{equation}
If $\alpha >\frac12 +\beta$, then the exponent $\kappa (\alpha,\beta)>0$, and the 
bound \eqref{eq:hc} is better than \eqref{eq:random}.
\end{theorem}

\begin{proof}
Let $V\subset \cX_n$ be the $k$-dimensional subspace, 
$k=\gamma n,\,0<\gamma<1$, consisting of all vectors $(x_1,\dots,x_n)$ with $x_i=0$ if $i>k$. Let $N=2^{n-k}=2^{\alpha n},\,\alpha =1-\gamma.$
The affine subspaces 
\begin{equation*}
 V_i=V+s_i, \quad s_i\in \cX_n /V
\end{equation*}
form a partition of the Hamming space
\begin{equation*}
\cX_n=\bigcup\nolimits^N_{i=1}V_i\, ,\quad V_i\cap V_j =\emptyset\, , 
\end{equation*}
where 
$|V_i|=2^{\gamma n}, \diam V_i = \gamma n,$ 
where $\diam\cE=\max\{d(x_1,x_2):x_1,x_2\in\cE\}$ denotes the diameter of a 
subset $\cE\subseteq\cX_n$.

We consider a subset $Z_N=\{z_i\}_1^N$ with $z_i\in V_i, i=1,\dots,N$.
For such a subset, the local discrepancy \eqref{eq:local} can be written as follows
   \begin{equation}\label{eq:sz}
       D(Z_N,y,t)=\sum\nolimits^N_{i=1} \zeta_i(y,t)\, ,
   \end{equation}
where
\begin{equation*}
    \zeta_i(y,t)=\bb_{\{B(y,t)\cap V_i\}}(z_i)-N\frac{ |(B(y,t)\cap V_i|}{|\cX_n|}.
\end{equation*}
Notice that if $V_i\subset B(y,t),$ then $\zeta_i(y,t)\equiv 0$ (recall that $x_i\in V_i$).
Therefore, the sum \eqref{eq:sz} takes the form
\begin{equation*}
D(Z_N,y,t)=\sum\nolimits^N_{i\in J} \zeta_i(y,t)\, ,
\end{equation*}
where $J$ is a subset of indices $i$ such that $V_i\cap B(y,t)\ne\emptyset$
but $V_i\not\subset B(y,t)$ ($V_i$ is not either completely inside or completely outside $B(y,t)$).
Since $\diam V_i=k$, we conclude that all  
$V_i,\, i\in J,$ are contained in the ball $B(y,t+k)$ and do not intersect the 
ball $B(y,t-k-1)$. Therefore,
\begin{equation*}
|J|\,|V_i|\le v(t+k)-v(t-k-1)\le v(t+k)\, .
\end{equation*}
Here we estimate $J$ from above by the number of sets $V_i$ such that $B(y,t)\subset V_i.$
We note that discarding the term $v(t-k-1)$ entails no significant loss in the asymptotics because this term
is exponentially small compared to $v(t+k)$. For $t\le\beta n$, using the bound \eqref{eq:H} and $\alpha+\gamma=1,$ we obtain
\begin{equation*}
|J|\le 2^{nH(\beta +\gamma)-\gamma n}=2^{\alpha n(1-\kappa)}=N^{1-\kappa}\, ,
\end{equation*}
where $\kappa$ is defined in \eqref{eq:kappa}.

Now consider a random subset $Z_N=\{z_i\}_1^N$ in which each point $z_i$ is selected independently and uniformly in $V_i$. 
For a subset $\cE\in V_i$ we have $\Pr(z_i\in\cE)=|\cE|/|V_i|=N|\cE|/|\cX_n|.$The quantities $\zeta_i(y,t)$ are bounded independent random 
variables that satisfy $|\zeta_i (y,t)|\le 1$ and $\EE\,\zeta_i (y,t)=0$.   Applying the Marcinkiewicz--Zygmund inequality to the sum \eqref{eq:sz},
we obtain 
\begin{equation*}
\EE\,|D(Z_N,y,t)|^p\le 2^p(p+1)^{\,p/2}\,
N^{\,p(1-\kappa)/2}\, 
\end{equation*}
and, therefore, in view of \eqref{eq:norm},
\begin{equation}\label{eq:21}
\EE\, |D(G,Z_N)|^p\le 2^p(p+1)^{\,p/2}\,
\, N^{\,p(1-\kappa)/2}.
\end{equation}
Thus, there exists a subset 
	$Z_N=Z_N(p)\subset \cX_n,\, 1\le p<\infty,$ whose discrepancy is bounded 
	above 
	as in this inequality. For $0<p<1$, in view of \eqref{eq:incr}, we can put 
	$Z(p)=Z(1)$ to complete the proof.
\end{proof}

{\em Remark 2.2.} We conjecture that the improvement of the discrepancy estimate for weights equal to zero in
the neighborhood of $t=n/2$ takes place also for $d$-dimensional Euclidean spheres $S^d\subset \reals^{d+1}$ in the case 
that the dimension $d$ grows in proportion to the cardinality $N$. Indeed, the sphere $S^d$ and the Hamming space $\cX_n$ share the
property that for large dimensions the invariant measure concentrates around the ``equator''. This interesting problem deserves a
separate study. 

%
%

\subsection{The case \texorpdfstring{${p=\infty}$}{}}
The following statement is analogous to
\cite[Prop.2.2]{skriganov2018bounds}. 
For $1\le p<\infty$ and any subset $Z_N\subseteq\cX_n$, we have 
\begin{equation}\label{eq:di1}
D_{\infty}(I,Z_N)\le |I|^{\,1/p}\,\, 2^{\,n/p}\,\, D_p(G_I,Z_N)\, ,	
\end{equation}
where
\begin{align*}
D_p(G_I,Z_N)=\Bigl(\,\sum\nolimits^{\nu}_{t=0} |I|^{-1}
\sum\nolimits_{y\in\cX_n}2^{-n} 
|D(Z_N,y,t)|^p\,\Bigr)^{1/p},
\end{align*}
is a special $L_p$-discrepancy with $G_I=(g_1,\dots,g_{\nu})$, where 
$g_t=|I|^{-1}$ for $t\in I$ and $g_t=0$ otherwise.

Indeed, for ${y}\in\cX_n$ and $t\in I$ we have
\begin{align*}
|D(Z_N,y,t)|&\le \Bigl(\,\sum\nolimits_{t\in I} 
\sum\nolimits_{y\in\cX_n} 
|D(Z_N,y,t)|^p\,\Bigr)^{1/p}
\notag
\\
&=|I|^{\,1/p}\,\, 2^{\,n/p}\,\Bigl(\,\sum\nolimits_{t\in I} |I|^{-1}
\sum\nolimits_{y\in\cX_n} 2^{-n}
|D(Z_N,y,t)|^p\,\Bigr)^{1/p}.
\end{align*}
\begin{theorem}
(i)	Let $I\subseteq\{0,1,\dots,n\}$ be an arbitrary subset of the set of radii,
and $N\le 2^{n-1}$. 
Then 
\begin{equation}\label{eq:Di01}
	D_{\infty}(I,Z_N)\le 8\,(1+n)^{\,1/2}\, N^{\,1/2}\, .
\end{equation}
If $N$ increases exponentially, $N\cong 2^{\alpha n}$, then 
$D_\infty(I,n,N)=O((\log_2 N)^{1/2}\, N^{1/2}).$

(ii) Let $I\subseteq\{0,1,\dots,\beta n\}$ be an arbitrary subset of the set of 
radii $t\le\beta n, 0<\beta<1/2$, and let $N=2^{\alpha n}\le 2^{n-1}$ be 
a power of $2$.
Then 
\begin{equation}\label{eq:Di02}
D_{\infty}(I,n,N)\le 8\,\Bigl(2+\frac{\log_2 N}{\alpha}\Bigr)^{1/2}\, 
N^{\,(1-\kappa)/2}\, ,	
\end{equation}
where the exponent $\kappa=\kappa (\alpha,\beta)$ is given in \eqref{eq:kappa}.	
If $\alpha >\frac12 +\beta$, then the exponent $\kappa (\alpha,\beta)>0$, and 
the 
bound \eqref{eq:Di02} is better than \eqref{eq:Di01}.
\end{theorem}
\begin{proof}
Substituting the bounds \eqref{eq:random} and \eqref{eq:hc} into inequality \eqref{eq:di1}, we obtain
   \begin{align}\label{eq:di2}
D_{\infty}(I,Z_N)\le n^{1/p}\,\, 2^{n/p}\,2\,(p+1)^{1/2}\, N^{1/2}
   \end{align}
and 
   \begin{align}\label{eq:di3}
D_{\infty}(I,Z_N)\le n^{\,1/p}\,\, 2^{\,n/p}\,2\,(p+1)^{\,1/2}\, 
N^{\,(1-\kappa)/2}\, .
   \end{align}
Now, we put $p=n$ in \eqref{eq:di2} and \eqref{eq:di3} to obtain, respectively, \eqref{eq:Di01} and \eqref{eq:Di02}.
\end{proof}

%
%
\section{Discrepancy for hemispheres}
Let $\cX_{n}, n=2m+1$ be the Hamming space.
In this section we consider a restricted version of discrepancy where instead of all the ball radii in \eqref{eq:Dp} we consider
discrepancy only with respect to the balls of radius $m$, calling them hemispheres. For any pair of antipodal points $y,\bar y$
   \begin{equation*}
   \cX_n=B(y,m)\cup B(\bar y,m)\, ,\quad B(y,m)\cap B(\bar y,m)=\emptyset\, ,
   \end{equation*}
hence $2^{-n}v(m)=2^{-n}|B(y,m)|=1/2.$   

For a subset $Z_N\subset \cX_n$ 
define
   \begin{equation}
   D_p^{(m)}(Z_N)=\Big(2^{-n}\sum\nolimits_{y\in \cX_n} 
   |D(Z_N,y,m)|^p\Big)^{1/p},\quad 0<p<\infty\, ,
   \label{eq:hsd}
   \end{equation}
where 
    $$
   D(Z_N,y,m)=|B(y,m)\cap Z_N|-\frac {{N}}2
    $$
is the local discrepancy defined in \eqref{eq:local}. In the previous notation 
$D_p^{(m)}(Z_N)=D_p(\,G(m),Z_N)$, with weights $G(m)=(g_1,\dots,g_{n-1}),$ where 
$g_m=1$ and $g_t=0$ if $t\ne m$.
Further, let
    $$
    D_\infty^{(m)}(Z_N)=\max_{y\in \cX_n}|D(Z_{{N}},y)|\, .
    $$
As before, define 
  $$
  D_p^{(m)}(n,N)=\min_{Z_N\subset \cX_n} D_p^{(m)}(Z_N),\quad p\in(0,\infty].
  $$

First we address the question of global minimizers of discrepancy.
\begin{theorem}\label{thm1}
	For the Hamming space $\cX_n$ with odd $n=2m+1$, we have the 
	following.

(i) Let $N=2K$ be even, then for all subsets $Z_N\subseteq \cX_n$ and 
$p\in (0, \infty]$
    \begin{equation}\label{eq:0}
D_p^{(m)}(Z_N)\ge 0
    \end{equation}
with equality for subsets $Z_N$ consisting of $K$ pairs of antipodal points.

(ii) Let $N=2K+1$ be odd, then for all subsets $Z_N\subseteq \cX_n$ and 
$p\in (0, \infty]$
    \begin{equation}\label{eq:1}
D_p^{(m)}(Z_N)\ge 1/2
    \end{equation}
with equality for subsets $Z_N$ consisting of $K$ pairs of antipodal points 
supplemented with a single point.

In other words, for all $p\in (0, \infty]$ the extremal discrepancy
$D_p^{(m)}(n,N)=0$	if $N$ is even and $D_p^{(m)}(n,N)=1/2$ if $N$ is odd.
\end{theorem}

{\em Remark 3.1.} The phenomenon of such small discrepancies for 
hemispheres is also known for Euclidean spheres $S^d\subset \reals^{d+1}$, see \cite{Bilyk2018, 
Skriganov2017, Skriganov2019}. The sphere $S^d$ can be represented as a 
disjoint union of two antipodal hemispheres and the equator. But the equator in 
this partition is of zero invariant measure and has no 
effect on the discrepancy. A similar situation holds for the Hamming space
$\cX_n$ with odd $n$, because in this case the ``equator'' with $t=n/2$ is simply an empty set.

\begin{proof}	
From \eqref{eq:hsd} we conclude that
  $$
N=|B(y,m)\cap Z_N|+|B(\bar y,m)\cap Z_N|\, ,
  $$
and for any $y\in \cX_n$ the local discrepancy  can be written as 
\begin{align}\label{eq:eo}
2|D(Z_N,y,m)|&=\Bigl|\,2\,|B(y,m)\cap Z_N|-N\,\Bigr|
=\Bigl|\,|B(y,m)\cap Z_N|-|B(\bar y,m)\cap Z_N|\,\Bigr|\, .
\end{align}


Let $N=2K$. Inequality \eqref{eq:0} holds for all subsets $Z_N$. If $Z_N$ is formed of $K$ pairs
of antipodal points, then $|D(Z_N,y,m)|=0$ for all $y\in\cX_n$.
This proves part (i).

Let $N=2K+1$. It follows from \eqref{eq:eo} that $2|D(Z_N,y,m)|\ge 1$, since $N$ 
is odd 
and $2\,|B(y,m)\cap Z_N|$ is even. This implies inequality \eqref{eq:1}. Furthermore, 
it also follows from \eqref{eq:eo} that $2|D(Z_N,y,m)|=1$ for all $y\in\cX_n$ 
if  
$Z_N$ consists of $K$ pairs of antipodal points supplemented with a single point. This proves part (ii).
\end{proof}

Thus in particular, any linear code $Z_N\subset\cX_n$ that contains the all-ones vector has discrepancy zero (such codes are called self-complementary). Many well-known families of binary linear codes such as the Hamming codes, BCH codes, etc. possess this property.

A minor generalization of the above proof implies the following useful relation. Let 
$Z_N=Z_N^{'}\cup Z^{''}_N$ be a union of two subsets, where $Z_N^{'}$ contains 
all 
pairs of antipodal points in $Z_N$ then
     $$
D^{(m)}_p(Z_N)=D^{(m)}_p(Z_N^{''})\, ,\quad p\in (0, \infty] .
     $$

\subsection{Quadratic discrepancy for hemispheres}
 In this section we consider the discrepancy $D_p^{(m)}(Z_N)$ defined in \eqref{eq:hsd} for the special case $p=2.$ Let $Z_N\subset \cX_n$ be a code, where $n=2m+1.$
 For a pair of points $x,y\in\cX_n$ such that $d(x,y)=w$ let $\mu_m(x,y)=\mu_m(w)=|B(x)\cap B(y)|$ be the size of the intersection of the
balls of radius $t$ with centers at $x$ and $y.$ By abuse of notation we write 
$\mu_m$ both as a kernel on $\cX_n\times\cX_n$ and as a function
on $\{0,1,\dots,n\}$. This is possible because $\mu_m(x,y)$ depends only on the distance between $x$ and $y$. Note that
$\mu_m(0)=v(m)=2^{n-1}$ and $\mu_m(n)=0.$

In this subsection we use some more specific facts of coding theory. We refer to \cite{mac91} for details.
For a code $Z_N\subset \cX_n$ let 
  $$
  A_w=A_w(Z_N)=\frac 1N|\{(z_i,z_j)\in Z_N^2\mid d(z_i,z_i)=w\}|, \quad w=0,1,\dots,n
  $$
   be the normalized number of ordered pairs of points at distance $w$ (the numbers $A_w,w=0,1,\dots,n$ form the {\em distance distribution} of $Z_N$).
Recall that the {\em dual distance distribution} of the code $Z_N$ is given by
  \begin{equation}\label{eq:MW}
  A^\bot_i=\frac 1N\sum\nolimits_{w=0}^n A_w K^{(n)}_i(w), \quad i=0,1,\dots,n,
  \end{equation}
where $K^{(n)}_i(x)$ be the binary Krawtchouk polynomial of degree $k=0,\dots,n$, defined as follows:
   \begin{equation}\label{eq:K}
   K^{(n)}_i(x)=\sum_{j=0}^i(-1)^j\binom xj\binom{n-x}{i-j}.
   \end{equation}  
The vector $(A^\bot_i)$ forms the MacWilliams transform of the distance distribution of the code $Z_N,$ and if $Z_N$ is a linear code,
it coincides with the weight distribution of the {\em dual code} $Z_N^\bot$ \cite[pp.~129,138]{mac91}. The MacWilliams transform is
an involution \cite[Thm.~3]{del72b}, which enables us to invert relations \eqref{eq:MW}:
  \begin{equation}\label{eq:MWI}
   A_i=\frac {2^n}{N}\sum\nolimits_{w=0}^n A_w^\bot K_i^{(n)}(w), \quad i=0,1,\dots,n.
   \end{equation}
The following result is implied by \cite{Barg2020}, Lemma 4.1.

 \begin{lemma} \label{lemma:mu} The Krawtchouk expansion of the function $\mu_m(w),w=0,1,\dots,n$ has the form
  \begin{equation*}
\mu_m(w)=\hat\mu_0+\sum_{\begin{substack}{k= 1\\ k\text{\rm \,odd}}\end{substack}}^n \hat\mu_kK^{(n)}_k(w)
  \end{equation*}
  where $\hat\mu_0=2^{n-2}$ and for all $k=1,3,\dots,n$
    $$
    \hat\mu_k=2^{-n}\binom{2m}m^2\frac{\binom{m}{(k-1)/2}^2}{\binom{2m}{k-1}^2}.
  $$
  \end{lemma}

In the next proposition we establish a version of Stolarsky's invariance principle for the quadratic discrepancy $D_2^{(m)}(Z_N)$ defined above in \eqref{eq:hsd}.
\begin{proposition} We have
    \begin{align}\label{eq:SI}
    2^n N^{-2} D_2^{(m)}(Z_N)^2& =\frac 1N \sum\nolimits_{w=0}^{n} A_w\mu_m(w)-2^{n-2}\\
      &=\sum_{\begin{substack}{k= 1\\ k\text{\rm 
  \,odd}}\end{substack}}^n 
  \hat\mu_kA_k^\bot. \label{eq:odd}
  \end{align}
\end{proposition}

\begin{proof}   Starting with \eqref{eq:hsd}, we compute 
  \begin{align}
 2^n D_2^{(m)}(Z_N)^2&=\sum\nolimits_{y\in\cX_n}\Big(\sum\nolimits_{j=1}^N 
  \bb_{B(y,m)}(z_j)-\frac N2\Big)^2=\sum\nolimits_{y\in\cX_n}\Big( 
  \sum\nolimits_{j=1}^N \bb_{B(z_j,m)}(y)-\frac N2\Big)^2
  \nonumber\\
  &=\sum\nolimits_{y\in \cX_n}\Big(\sum\nolimits_{i,j=1}^N 
  \bb_{B(z_i,m)}(y)\bb_{B(z_j,m)}(y)-N\sum\nolimits_{j=1}^N 
  \bb_{B(z_j,m)}(y)+\frac{N^2}{4}\Big)\nonumber\\
  &=\sum\nolimits_{i,j=1}^N\sum\nolimits_{y\in 
  \cX_n}\bb_{B(z_i,m)}(y)\bb_{B(z_j,m)}(y)-2^{n-2} N^2\nonumber\\
  &=\sum\nolimits_{i,j=1}^N \mu_m(z_i,z_j)-2^{n-2} N^2=N\sum\nolimits_{w=0}^n A_w\mu_m(w)-2^{n-2}N^2,\nonumber
  \end{align}
where the last equality uses the definition of $A_w.$ This proves \eqref{eq:SI}. 

To obtain \eqref{eq:odd}, substitute the result of Lemma~\ref{lemma:mu} 
into \eqref{eq:SI} and then use \eqref{eq:MW}. 
\end{proof}

The size of the intersection of the balls can be written in a more explicit form:
    $$
    \mu_m(w)=\sum\nolimits_{i,j} \binom wi \binom{n-w}j, \quad w=0,1,\dots,n,
    $$
where $i+j\le m, 0\le w-i+j\le m;$ in particular, $\mu_m(0)=2^{n-1}.$ It is not difficult to show that for any $l=1,2,\dots,\lfloor n/2\rfloor$
 we have $
  \mu_m(2l-1)=\mu_m(2l)
  $
and otherwise $\mu_m(w)$ is a decreasing function of $w$.

  Let $\langle \mu_m\rangle_{\cE}$ be the average value of the kernel $\mu_m(x,y)$ over the subset $\cE\subset \cX_n$.
Since $\langle \mu_m\rangle_{\cX_n}=\hat\mu_0,$ we can write  \eqref{eq:SI} in the following form:
  \begin{equation}\label{eq:D2}
 2^n N^{-2}\, D_2^{(m)}(Z_N)^2=\langle 
 \mu_m\rangle_{Z_N}-\langle\mu_m\rangle_{\cX_n}.
  \end{equation}
Relations \eqref{eq:D2}, \eqref{eq:SI} are similar to the invariance principle for hemispheres in the case of the Euclidean sphere, \cite[Thm.~3.1]{Bilyk2018}. At the same time, the concrete forms of the results for the Hamming space and the sphere are different: while for the sphere the quadratic discrepancy is expressed via the average geodesic distance in $Z_N$, in the Hamming case it is related to the average of the kernel $\mu_m$ and is not immediately connected to the average distance. 
Note that for quadratic discrepancy $D_2(G,Z_N)$ for the Hamming space defined above in \eqref{eq:Dp}, results of this form were previously established in \cite{Barg2020}.

Our final result in this section concerns a characterization of codes with zero discrepancy for hemispheres for the case of even $N$.

\begin{theorem} Let $Z_N$ be a code of even size $N$. Then $D_2^{(m)}(Z_N)=0$ if and only if the code $Z_N$ is formed of $N/2$ antipodal
pairs of points.
\end{theorem}
\begin{proof}
The sufficiency part has been proved in Theorem \ref{thm1}. The proof in the other direction is a combination of the following steps.

{\bf Step 1.} Since $\hat \mu_k>0$ for all $k,$ expression \eqref{eq:odd} implies that a code $Z_N\subset \cX_n$ has zero quadratic discrepancy
for hemispheres if and only if its dual distance coefficients $A_k^\bot\ne 0$ only if $k$ is even,

{\bf Step 2.} A code $Z_N$ is formed of antipodal pairs if and only if its distance distribution is symmetric, i.e., $A_w=A_{n-w}$ for all $w=0,1,\dots,m.$

Indeed, the distance distribution coefficients $A_w, w=0,\dots,n$ can be written as
      \begin{equation}\label{eq:local1}
      A_w=\sum_{z\in Z_N}A_w(z),
      \end{equation}
where       
   $
   A_w(z)=\frac 1N|\{y: d(z,y)=w\}|
   $
is the local distance distribution at the point $z\in Z_N.$

Suppose the code is formed of antipodal pairs. For every $y\in Z_N$ such that $d(z,y)=w$, the opposite point $\bar y$ satisfies $d(z,\bar y)=n-w,$
and thus, the pair $(y,\bar y)$ contributes to $A_w(z)$ and $A_{n-w}(z)$ in equal amounts. Therefore, from \eqref{eq:local1} also $A_w=A_{n-w}.$

Now suppose that the distance distribution is symmetric. For any code $A_0=1,$ and then also $A_n=1,$ but this means that
every code point has a diametrically opposite one, or otherwise \eqref{eq:local1} cannot be satisfied for $w=n.$

{\bf Step 3.} The matrix 
  $$
    \Phi_m=\begin{pmatrix} K^{(n)}_1(0)&K^{(n)}_1(1)&\dots&K^{(n)}_1(m)\\K^{(n)}_3(0)&K^{(n)}_3(1)&\dots&K^{(n)}_3(m)\\\vdots&\vdots &\dots&\vdots\\
    K^{(n)}_{2m+1}(0)&K^{(n)}_{2m+1}(1)&\dots&K^{(n)}_{2m+1}(m)\end{pmatrix}
    $$ 
has rank $m+1.$ This is shown as follows. Orthogonality of Krawtchouk polynomials \cite{del72b}, \cite[Thm 5.16]{mac91} implies that
\begin{align*}
  \binom{n}{k}2^n\delta_{j,k}&=\sum_{w=0}^{2m+1}K^{(n)}_k(w)K^{(n)}_j(w)\binom nw\\
  &=\sum_{w=0}^mK^{(n)}_k(w)K^{(n)}_j(w)\binom nw+\sum_{w=m+1}^{2m+1} (-1)^{j+k}
  K^{(n)}_k(n-w)K^{(n)}_j(n-w)\binom n{n-w}\\
  &=2\sum_{w=0}^mK^{(n)}_k(w)K^{(n)}_j(w)\binom nw.
  \end{align*} 
  Here on the second line we used the relation
  \begin{equation}\label{eq:sym}
  K^{(n)}_k(w)=(-1)^kK^{(n)}_k(n-w), \quad0\le k,w\le n.
  \end{equation}
which is immediate from \eqref{eq:K}.
In other words, for odd $j,k$ we have
   \begin{equation}\label{eq:litor}
   \sum_{w=0}^m K^{(n)}_k(w)K^{(n)}_j(w)\binom nw=\delta_{k,j}2^{n-1}\binom nk.
   \end{equation}
Rephrasing this relation, we obtain 
   $$
   \Phi_m B\Phi_m^T=2^{n-1}\text{diag}\Big(\binom n1,\binom n3,\dots,\binom n{2m+1}\Big),
   $$
   where $B=\text{diag}(\binom nw, w=0,1,\dots,m).$
This implies that $\rank(\Phi_m)=m.$  
  
{\bf Step 4.} To complete the proof, 
suppose that $D_2^{(m)}(Z_N)=0$ and thus from Step 1 above, $A_k^\bot=0$ for all odd $k$. 
In particular, for $k=1,3,\dots,2m+1,$ using \eqref{eq:MW} and \eqref{eq:sym}, we obtain
   \begin{align}\label{eq:rank}
   \sum_{w=0}^{2m+1}A_w K^{(n)}_k(w)&=\sum_{w=0}^m (A_w-A_{n-w}) K^{(n)}_k(w)=0.
    \end{align}
Define the vector
  $$
  \alpha=(A_w-A_{n-w}, w=0,1,\dots,m).
  $$
From \eqref{eq:rank} and the definition of $\Phi_m$ we obtain that $\Phi_m\alpha^T=0.$ From Step 3), we conclude that $\alpha=0$ or $A_w=A_{n-w}, w=0,1,\dots,m.$  Now Step 2 implies our claim.
\end{proof}


\begin{thebibliography}{10}
\providecommand{\url}[1]{#1}
\csname url@samestyle\endcsname
\providecommand{\newblock}{\relax}
\providecommand{\bibinfo}[2]{#2}
\providecommand{\BIBentrySTDinterwordspacing}{\spaceskip=0pt\relax}
\providecommand{\BIBentryALTinterwordstretchfactor}{4}
\providecommand{\BIBentryALTinterwordspacing}{\spaceskip=\fontdimen2\font plus
\BIBentryALTinterwordstretchfactor\fontdimen3\font minus
  \fontdimen4\font\relax}
\providecommand{\BIBforeignlanguage}[2]{{%
\expandafter\ifx\csname l@#1\endcsname\relax
\typeout{** WARNING: IEEEtranS.bst: No hyphenation pattern has been}%
\typeout{** loaded for the language `#1'. Using the pattern for}%
\typeout{** the default language instead.}%
\else
\language=\csname l@#1\endcsname
\fi
#2}}
\providecommand{\BIBdecl}{\relax}
\BIBdecl

\bibitem{Barg2020}
A.~Barg, ``Stolarsky's invariance principle for finite metric spaces,''
  preprint, arXiv:2005.12995, May 2020.

\bibitem{Beck1984}
J.~Beck, ``Sums of distances between points on a sphere---an application of the
  theory of irregularities of distribution to discrete geometry,''
  \emph{Mathematika}, vol.~31, no.~1, pp. 33--41, 1984.

\bibitem{Beck1987}
J.~Beck and W.~W.~L. Chen, \emph{Irregularities of Distribution}.\hskip 1em
  plus 0.5em minus 0.4em\relax Cambridge, UK: Cambridge University Press, 1987.

\bibitem{Bilyk2018}
D.~Bilyk, F.~Dai, and R.~Matzke, ``The {S}tolarsky principle and energy
  optimization on the sphere,'' \emph{Constr. Approx.}, vol.~48, pp. 31--60,
  2018.

\bibitem{BilykLacey2017}
D.~Bilyk and M.~Lacey, ``One-bit sensing, discrepancy and {S}tolarsky's
  principle,'' \emph{Sbornik: Mathematics}, vol. 208, no.~6, pp. 744--763,
  2017, translation from the Russian, {\em Mat. Sbornik}, vol. 208, no.~6,
  pp.4--25.

\bibitem{BCCGT2019}
L.~Brandolini, W.~W.~L. Chen, L.~Colzani, G.~Gigante, and G.~Travaglini,
  ``Discrepancy and numerical integration on metric measure spaces,'' \emph{J.
  Geom. Anal.}, vol.~29, no.~1, pp. 328--369, 2019.

\bibitem{BGKZ2020}
J.~S. Brauchart, P.~J. Grabner, W.~B. Kusner, and J.~Ziefle, ``Hyperuniform
  point sets on the sphere: probabilistic aspects,'' \emph{Monatshefte f{\"u}r
  {M}athematik}, vol. 192, pp. 763–--781, 2020.

\bibitem{BG2015}
J.~S. Brauchart and P.~J. Grabner, ``Distributing many points on spheres:
  minimal energy and designs,'' \emph{J. Complexity}, vol.~31, no.~3, pp.
  293--326, 2015.

\bibitem{Brauchart2013}
J.~Brauchart and J.~Dick, ``A simple proof of {S}tolarsky's invariance
  principle,'' \emph{Proc. Amer. Math. Soc.}, vol. 141, pp. 2085--2096, 2013.

\bibitem{CT1988}
Y.~Chow and H.~Teicher, \emph{Probability Theory: Independence,
  Interchangeability, Martingales}, 2nd~ed.\hskip 1em plus 0.5em minus
  0.4em\relax Springer, 1988.

\bibitem{del72b}
P.~Delsarte, ``Bounds for unrestricted codes, by linear programming,''
  \emph{Philips Res. Rep.}, vol.~27, pp. 272--289, 1972.

\bibitem{lev98}
V.~I.~Levenshtein,
``Universal bounds for codes and designs,'' in \emph{Handbook of Coding 
Theory} (V.~S.~Pless and W.~C.~Huffman eds.), Vol. 1, Chapter~6, pages 499--648, 
Elsevier, 1998.

\bibitem{mac91}
F.~J. MacWilliams and N.~J.~A. Sloane, \emph{The Theory of Error-Correcting
  Codes}.\hskip 1em plus 0.5em minus 0.4em\relax Amsterdam: North-Holland,
  1991.

\bibitem{Matousek1999}
J.~Matou{\v s}ek, \emph{Geometric Discrepancy: {A}n Illustrated Guide}.\hskip
  1em plus 0.5em minus 0.4em\relax Berlin-Heidelberg: Springer-Verlag, 1999.

\bibitem{Skriganov2017}
M.~M. Skriganov, ``Point distributions in compact metric spaces,''
  \emph{Mathematika}, vol.~63, no.~3, pp. 1152--1171, 2017.

\bibitem{Skriganov2019}
------, ``Point distributions in two-point homogeneous spaces,''
  \emph{Mathematika}, vol.~65, no.~3, pp. 557--587, 2019.

\bibitem{skriganov2018bounds}
------, ``Bounds for {$L_p$}-discrepancies of point distributions in compact
  metric spaces,'' \emph{Constr. Approx.}, vol.~51, pp. 413--425, 2020.

\bibitem{Skriganov2020}
------, ``Stolarsky's invariance principle for projective spaces,'' \emph{J.
  Complexity}, vol.~56, p. 101428, 2020.

\bibitem{Stolarsky1973}
K.~B. Stolarsky, ``Sums of distances between points on a sphere, {II},''
  \emph{Proc. AMS}, vol.~41, pp. 575--582, 1973.

\end{thebibliography}
\end{document}